\documentclass[11pt, a4paper]{amsart}

\usepackage[utf8]{inputenc}
\usepackage[T1]{fontenc}
\usepackage[english]{babel}
\usepackage{amsmath}
\usepackage{amsfonts}
\usepackage{ae}
\usepackage{units}
\usepackage{icomma}
\usepackage{color}
\usepackage{graphicx}
\usepackage{epstopdf}
\usepackage{amsthm}
\usepackage{amssymb}
\usepackage{cancel}
\usepackage{fullpage}
\usepackage{upgreek}
\usepackage{type1cm}
\usepackage{eso-pic}
\usepackage[breaklinks,pdfpagelabels=false]{hyperref}

\renewcommand{\epsilon}{\varepsilon}

\newtheorem{theorem}{Theorem}[section]
\newtheorem{lemma}[theorem]{Lemma}

\theoremstyle{definition}

\newtheorem{question}[theorem]{Question}

\numberwithin{equation}{section}
\numberwithin{theorem}{section}

\begin{document}
\setcounter{figure}{0}

\title[Permutations destroying arithmetic progressions in finite cyclic groups] 
{Permutations destroying arithmetic progressions in finite cyclic groups}

\author{Peter Hegarty$^{1,2}$ \and Anders Martinsson$^{1,2}$} 
\address{$^1$Mathematical Sciences, Chalmers, 41296 Gothenburg, Sweden} 
\address{$^2$Mathematical Sciences, University of Gothenburg,  41296 Gothenburg, Sweden} 
\email{hegarty@chalmers.se}

\email{andemar@chalmers.se}



\subjclass[2000]{11B75} \keywords{Permutation, arithmetic progression, finite cyclic group}

\date{\today}

\begin{abstract} 
A permutation $\pi$ of an abelian group $G$ is said to destroy arithmetic progressions (APs) if, whenever $(a, \, b, \, c)$ is a non-trivial 3-term AP in $G$, that is $c-b=b-a$ and $a, \, b, \, c$ are not all equal, then $(\pi(a), \, \pi(b), \pi(c))$ is not an AP. In a paper from 2004, the first author conjectured that such a permutation exists of $\mathbb{Z}_n$, for all $n \not\in \{2, \, 3, \, 5, \, 7\}$. Here we prove, as a special case of a more general result, that such a permutation exists for all $n \geq n_0$, for some explcitly constructed number $n_0 \approx 1.4 \times 10^{14}$. We also construct such a permutation of $\mathbb{Z}_p$ for all primes $p > 3$ such that $p \equiv 3 \; ({\hbox{mod $8$}})$.
\end{abstract}

\maketitle



\setcounter{section}{0} 

\setcounter{equation}{0} 

\section{Introduction}\label{sect:intro}

Let $G$ be an abelian group, $S$ a subset of $G$. A bijection $\pi: S \rightarrow S$ is said to \emph{destroy}{\footnote{The term \emph{avoid} was used in \cite{H}. The term \emph{destroy} was preferred in \cite{JS}, and we agree it captures the notion better.}} \emph{arithmetic progressions (APs)} if there is no triple $(a, \, b, \, c)$ of elements of $S$ satisfying
\par (i) $a, \, b, \, c$ are not all equal,
\par (ii) $c-b = b-a$, 
\par (iii) $\pi(c) - \pi(b) = \pi(b) - \pi(a)$. 
\\
\par This notion was introduced by the first author in \cite{H}, though earlier Sidorenko \cite{S} had given an example of such a permutation in the case $G = \mathbb{Z}$, $S = \mathbb{N}$. It should not be confused with the somewhat different, and probably more famous, notion of a permutation \emph{containing} no arithmetic progressions \cite{DEGS}. 

The most important open question from \cite{H} concerns the existence of AP-destroying permutations of finite cyclic groups $\mathbb{Z}_n$. Conjecture C of that paper asserts that such permutations exist if and only if $n \not\in \{2, \, 3, \, 5, \, 7\}$. In this paper we come close to resolving this conjecture in full. Before stating our main result, we need to define an extension of the concept of AP-destroying permutation, in the special case of finite cyclic groups:
\\
\\
{\sc Definition 1.1.} Let $s, \, t \in \mathbb{N}_0$. A permutation $\pi$ of $\mathbb{Z}_n$ is said to \emph{destroy $(s,\,t)$-almost APs} if there is no triple $(a, \, b, \, c)$ of elements of $\mathbb{Z}_n$ satisfying
\par (i) $a, \, b, \, c$ are not all equal, 
\par (ii) $a+c-2b \equiv \eta_1 \; ({\hbox{mod $n$}})$ for some $\eta_1 \in \{0, \, \pm 1, \, \dots, \, \pm s\}$, 
\par (iii) $\pi(a) + \pi(c) - 2\pi(b) \equiv \eta_2 \; ({\hbox{mod $n$}})$ for some $\eta_2 \in \{0, \, \pm 1, \, \dots, \, \pm t\}$.
\\
\par Hence $(s,\,t) = (0,\,0)$ is the case of an AP-destroying permutation. We shall prove

\begin{theorem}\label{thm:main}
(i) There exists a permutation of $\mathbb{Z}_n$ destroying arithmetic progressions for all $n \geq n_0$ where 
\begin{equation}\label{eq:bignumber}
n_0 = \left( 9 \times 11 \times 16 \times 17 \times 19 
\times 23 \right)^2
\approx 1.4 \times 10^{14}.
\end{equation}
(ii) For every $s, \, t \in \mathbb{N}_0$ there is an $n_{0}(s, \, t)$ such that there exists a permutation of $\mathbb{Z}_n$ destroying $(s, \, t)$-almost APs for all $n \geq n_{0}(s, \, t)$. 
\end{theorem}

While (i) may seem to be a special case of (ii), we state it separately for two reasons. First and foremost, we use (i) in the proof of (ii). Secondly, we have in this special case tried to find the best constant $n_{0}$ which our method will yield. Of course, we expect that Conjecture C of \cite{H} is true, but some additional ideas will probably be needed to prove it in full. 
\par The idea for the proof of (i) is to construct a ``master permutation'' of $\mathbb{Z}_{\sqrt{n_0}}$ which is $(1, \, 2)$-almost AP-avoiding, and then combine this with ideas from Proposition 2.3(ii) and Lemma 3.5 of \cite{H} to construct AP-destroying permutations of $\mathbb{Z}_n$ for all $n \geq n_0$. In our proof, we break down the $(1, \, 2)$-destroyal property into a total of $14$ simpler ones and find, by simple computer search, permutations of $6$ cyclic groups of pairwise relatively prime orders satisfying different subsets of these simpler properties. Finally, the master permutation is obtained via an application of the Chinese Remainder Theorem. The proof of (ii) follows a similar strategy, but this time the ``master permutation'' destroys $(2, \, 2)$-almost APs, and it requires a more subtle application of the aforementioned ideas from 
\cite{H} to get the final result. The full proof of Theorem \ref{thm:main} 
is presented in Section \ref{sect:pfmain}.

The values of $n_0(s, \, t)$ arising from our proof will be extremely large. Though we have tried to optimise the value which our method gives for $n_{0}(0, \, 0)$, it remains completely impractical to attempt to complete the proof of Conjecture C of \cite{H} by a brute-force computer search. The main point of our result is that we think it removes any substantial doubt whether the conjecture is true. We will expand on this issue in the final section of the paper. However, it remains interesting to try to prove the full conjecture and, in particular, to try to do so without resorting to any large-scale computer searches. It follows from Lemma 3.5 of \cite{H} that, if we let $\mathcal{P}$ denote the set of those $n \in \mathbb{N}$ for which $\mathbb{Z}_n$ admits an AP-destroying permutation, then $\mathcal{P}$ is closed under multiplication. Hence, a natural strategy is to first focus on primes. The following result will be proven in Section \ref{sect:pfprimes}:

\begin{theorem}\label{thm:primes}
Let $p$ be a prime such that $p > 3$ and $p \equiv 3 \; ({\hbox{mod $8$}})$. Then there exists a permutation of $\mathbb{Z}_p$ destroying arithmetic progressions.
\end{theorem}

As we shall see, there are obvious ways one could try to tinker with this proof so as to make it work also for other primes. So far, however, we have not found any such tinkering that works. This and other outstanding issues will be addressed in Section \ref{sect:final}.
  
\setcounter{table}{0}
\setcounter{equation}{0}

\section{Proof of Theorem \ref{thm:main}}\label{sect:pfmain}

We introduce some further notation. Let $s, \, t \in \mathbb{N}_0$. A permutation $\pi$ of $\mathbb{Z}_n$ is said to \emph{destroy the pattern $s \mapsto t$} if there is no triple $(a, \, b, \, c)$ of elements of $\mathbb{Z}_n$ satisfying
\par (i) $a, \, b, \, c$ are not all equal, 
\par (ii) $a+c-2b \equiv s \; ({\hbox{mod $n$}})$, 
\par (iii) $\pi(a) + \pi(c) - 2 \pi(b) \equiv t \; ({\hbox{mod $n$}})$.
\\
\\
Hence, $\pi$ destroys $(s, \, t)$-almost APs if and only if it destroys the patterns
$s^{\prime} \mapsto t^{\prime}$, for all $s^{\prime} \in [-s, \, s]$ and $t^{\prime} \in [-t, \, t]$. In the following assertions, $\xi^{-1}$ denotes the inverse of $\xi$ modulo $n$. The proofs are almost trivial:

\begin{lemma}\label{lem:normalise}
(i) Suppose $\pi: \mathbb{Z}_n \rightarrow \mathbb{Z}_n$ is a permutation destroying the pattern $0 \mapsto 1$ and that GCD$(t, \, n) = 1$. Then $\pi_{1}: \mathbb{Z}_n \rightarrow \mathbb{Z}_n$ given by 
\begin{equation}\label{eq:zerotoone}
\pi_{1} (x) = t \pi (x)
\end{equation}
is a permutation destroying the pattern $0 \mapsto t$. 
\\
(ii) Suppose $\pi: \mathbb{Z}_n \rightarrow \mathbb{Z}_n$ is a permutation destroying the pattern $1 \mapsto 1$ and that GCD$(s, \, n) =$ GCD$(t, \, n) = 1$. Then $\pi_2: \mathbb{Z}_n \rightarrow \mathbb{Z}_n$ given by
\begin{equation}\label{onetoone}
\pi_{2} (x) = t \pi (s^{-1} x)
\end{equation}
is a permutation destroying the pattern $s \mapsto t$. 
\\
(iii) If $\pi: \mathbb{Z}_n \rightarrow \mathbb{Z}_n$ is a permutation destroying the pattern $s \mapsto t$, then $\pi^{-1}$ destroys $t \mapsto s$. 
\end{lemma}

\par The reader is encouraged to write their own program to check the correctness of the data in Table \ref{tab:data}, which was obtained by computer search. Note that we are here identifying $\mathbb{Z}_n$ with the set $\{0,\,1,\,\dots,\,n-1\}$, and following standard practice in identifying the string $(a_0,\, a_1,\,\dots,\,a_{n-1})$ with the permutation $\pi: i \mapsto a_i$. 

\begin{table}[ht!]
\begin{center}
\begin{tabular}{|c|c|c|c|} \hline 
$i$ & $n_i$ & $\pi: \mathbb{Z}_{n_i} \rightarrow \mathbb{Z}_{n_i}$ & ${\hbox{Patterns destroyed by $\pi$,}}$ 
\\ $\;$ & $\;$ & $\;$ & ${\hbox{together with $0 \mapsto 0$}}$ \\ \hline \hline
$1$ & $9$ & $(0,\,1,\,8,\,3,\,2,\,6,\,4,\,7,\,5)$ & $0 \mapsto 2, \;\; -1 \mapsto -2$ \\ \hline
$2$ & $11$ & $(0,\,1,\,8,\,10,\,6,\,9,\,5,\,7,\,3,\,2,\,4)$ & $0 \mapsto -2, \;\; -1 \mapsto 2$ \\ \hline
$3$ & $16$ & $(0,\,2,\,5,\,3,\,15,\,12,\,1,\,14,\,10,\,8,\,11,\,13,\,4,\,7,\,6,\,9)$ & $1 \mapsto 1, \;\; -1 \mapsto -1$ \\ \hline
$4$ & $17$ & $(0,\,1,\,3,\,9,\,11,\,7,\,4,\,8,\,15,\,12,\,16,\,10,\,14,\,5,\,2,\,13,\,6)$ & $-1 \mapsto 1, \;\; 1 \mapsto -1$ \\ \hline
$5$ & $19$ & $(0,\,2,\,14,\,4,\,10,\,17,\,9,\,13,\,18,\,3,\,6,\,15,$ & $1 \mapsto 1, \;\; -1 \mapsto 1$ \\ 
$\;$ & $\;$ & $8,\,12,\,5,\,1,\,7,\,11,\,16)$ & $\;$ \\ \hline
$6$ & $23$ & $(0,\,1,\,4,\,3,\,21,\,22,\,2,\,11,\,12,\,7,\,8,\,5,$ & $0 \mapsto 1, \;\; 1 \mapsto 0,$ \\ 
$\;$ & $\;$ & $10,\,9,\,6,\,19,\,16,\,15,\,20,\,17,\,18,\,13,\,14)$ & $0 \mapsto -1, \;\; -1 \mapsto 0$ \\ \hline
$7$ & $25$ & $(0,\,2,\,5,\,1,\,3,\,9,\,13,\,20,\,10,\,15,\,23,\,4,\,21,$ & $1 \mapsto 1, \;\; -1 \mapsto 1$ \\
$\;$ & $\;$ & $17,\,24,\,7,\,22,\,18,\,12,\,16,\,19,\,8,\,14,\,6,\,11)$ & $\;$ \\ \hline
$8$ & $29$ & $(0,\,2,\,1,\,3,\,6,\,5,\,7,\,4,\,13,\,12,\,8,\,10,$ & $1 \mapsto 1$ \\ 
$\;$ & $\;$ & $9,\,24,\,16,\,14,\,20,\,18,\,25,\,23,\,27,\,26,\,28,\,17,$ & $\;$ \\
$\;$ & $\;$ & $15,\,21,\,11,\,19,\,22)$ & $\;$ \\ \hline
$9$ & $31$ & $(0,\,2,\,1,\,3,\,6,\,5,\,7,\,4,\,13,\,12,\,8,\,10,$ & $1 \mapsto 1$ \\
$\;$ & $\;$ & $9,\,11,\,14,\,20,\,27,\,23,\,25,\,24,\,26,\,29,\,28,\,30,$ & $\;$ \\
$\;$ & $\;$ & $16,\,18,\,17,\,19,\,22,\,21,\,15)$ & $\;$ \\ \hline
$10$ & $37$ & $(0,\,2,\,1,\,3,\,6,\,5,\,7,\,4,\,13,\,12,\,8,\,10,\,9,$ & $1 \mapsto 1$ \\ 
$\;$ & $\;$ & $11,\,14,\,18,\,15,\,17,\,21,\,24,\,22,\,32,\,31,\,35,\,30,$ & $\;$ \\ 
$\;$ & $\;$ & $33,\,19,\,34,\,36,\,23,\,20,\,27,\,25,\,29,\,26,\,28,\,16)$ & $\;$ \\ \hline
$11$ & $41$ & $(0,\,2,\,1,\,3,\,6,\,5,\,7,\,4,\,13,\,12,\,8,\,10,$ & $1 \mapsto 1$ \\ 
$\;$ & $\;$ & $9,\,11,\,14,\,18,\,15,\,17,\,21,\,23,\,22,\,25,\,29,\,35,$ & $\;$ \\
$\;$ & $\;$ & $38,\,36,\,31,\,34,\,40,\,19,\,37,\,39,\,16,\,27,\,26,\,28,$ & $\;$ \\
$\;$ & $\;$ & $32,\,24,\,33,\,30,\,20)$ & $\;$ \\ \hline
$12$ & $43$ & $(0,\,2,\,1,\,3,\,6,\,5,\,7,\,4,\,13,\,12,\,8,\,10,$ & $1 \mapsto 1$ \\
$\;$ & $\;$ & $9,\,11,\,14,\,18,\,15,\,17,\,21,\,23,\,22,\,19,\,26,\,35,$ & $\;$ \\
$\;$ & $\;$ & $41,\,36,\,39,\,34,\,16,\,33,\,40,\,38,\,37,\,27,\,24,\,20,$ & $\;$ \\
$\;$ & $\;$ & $28,\,42,\,25,\,31,\,29,\,32,\,30)$ & $\;$ \\ \hline
$13$ & $47$ & $(0,\,2,\,1,\,3,\,6,\,5,\,7,\,4,\,13,\,12,\,8,\,10,$ & $1 \mapsto 1$ \\
$\;$ & $\;$ & $9,\,11,\,14,\,18,\,15,\,17,\,21,\,23,\,22,\,19,\,26,\,20,$ & $\;$ \\
$\;$ & $\;$ & $31,\,16,\,29,\,39,\,41,\,44,\,37,\,43,\,24,\,45,\,38,\,28,$ & $\;$ \\
$\;$ & $\;$ & $46,\,25,\,33,\,27,\,34,\,30,\,40,\,42,\,36,\,32,\,35)$ & $\;$ \\ \hline
$14$ & $13$ & $(0,\,1,\,4,\,2,\,7,\,6,\,12,\,9,\,11,\,8,\,3,\,5,\,10)$ & $0 \mapsto 1$ \\ \hline
$15$ & $49$ & $(0,\,1,\,4,\,2,\,3,\,6,\,7,\,12,\,5,\,8,\,9,\,15,\,11,$ & $0 \mapsto 1$ \\
$\;$ & $\;$ & $13,\,10,\,16,\,14,\,21,\,20,\,22,\,28,\,17,\,25,\,18,\,19,$ & $\;$ \\
$\;$ & $\;$ & $23,\,24,\,35,\,38,\,40,\,37,\,43,\,44,\,48,\,45,\,41,\,42,$ & $\;$ \\
$\;$ & $\;$ & $31,\,47,\,46,\,26,\,32,\,36,\,27,\,30,\,29,\,39,\,34,\,33)$ & $\;$ \\ \hline
\end{tabular} 
\end{center}
\vspace{0.3cm}
\caption{}
\label{tab:data}
\end{table}
For each $i \in \{1,\,2,\,3,\,4,\,6\}$, let $\pi_i$ be the permutation of $\mathbb{Z}_{n_i}$ given in Table \ref{tab:data}. Let $\pi_5$ be the permutation of $\mathbb{Z}_{19}$ given by $\pi_{5}(x) = 2 \pi^{-1}(x)$, where $\pi$ is as in Table \ref{tab:data}, and observe that, by Lemma \ref{lem:normalise}, $\pi_{5}$ destroys the patterns $0 \mapsto 0$, $1 \mapsto 2$, $1 \mapsto -2$. Thus, for each of the $14$ non-zero pairs $(s_i, \, t_i) \in \{-1,\, 0, \, 1\} \times \{-2, \, -1, \, 0, \, 1, \, 2\}$, there is some $i \in [1, \, 6]$ such that $\pi_i$ destroys the pattern $s_i \mapsto t_i$. Let $\sigma: \mathbb{Z}_{\sqrt{n_0}} \rightarrow \prod_{i=1}^{6} \mathbb{Z}_{n_i}$ be the natural isomorphism of abelian groups given by the Chinese Remainder Theorem, i.e.: $\sigma( x \; ({\hbox{mod $\sqrt{n_0}$}})) = \prod_{i=1}^{6} (x \; ({\hbox{mod $n_i$}}))$. We claim that the map $\pi_0: \mathbb{Z}_{\sqrt{n_0}} \rightarrow \mathbb{Z}_{\sqrt{n_0}}$ given by 
\begin{equation}\label{eq:master}
\pi_0 = \sigma^{-1} \circ (\pi_1, \, \dots, \, \pi_{6}) \circ \sigma
\end{equation}
is a permutation destroying $(1,\, 2)$-almost APs. This will be our master permutation for the proof of Theorem \ref{thm:main}(i).

We now show that $\pi_0$ has the desired property. Let $s \in \{0, \pm 1\}$ and let $(a, b, c)$ be a non-trivial (that is, $a, b, c$ are not all equal) triple of elements in $\mathbb{Z}_{\sqrt{n_0}}$ such that $a+c-2b \equiv s\;(\hbox{mod $\sqrt{n_0}$})$. Let $(a_i, \, b_i, c_i)$, $i = 1,\,\dots,\,6$, be the projections on the various factors of this triple, after applying $\sigma$. We note that $a_i + c_i - 2b_i \equiv s \; ({\hbox{mod $n_i$}})$, and
\begin{equation}\label{eq:modmaster}
\pi_0(a) + \pi_0(c)-2\pi_0(b) \equiv \pi_i(a_i) + \pi_i(c_i)-2\pi_i(b_i)\,(\hbox{mod $n_i$})
\end{equation}
for every $i$. Since $(a, b, c)$ is non-trivial, there must be at least one factor, $i_1$ say, such that $(a_{i_1}, b_{i_1}, c_{i_1})$ is non-trivial. We consider two cases:
\\
\\
{\sc Case 1:} There is some $i_2$ such that $a_{i_2}=b_{i_2}=c_{i_2}$. 
\\
\\
Clearly, this can only occur if $s=0$, so $(a, b, c)$ and all its projections are APs. As $\pi_{i_1}$ is AP-destroying, we have $\pi_{i_1}(a_{i_1})+\pi_{i_1}(c_{i_1})-2 \pi_{i_1}(b_{i_1}) \not\equiv 0\;(\hbox{mod $n_{i_1}$})$. Furthermore, we trivially have $\pi_{i_2}(a_{i_2})+\pi_{i_2}(c_{i_2})-2 \pi_{i_2}(b_{i_2}) \equiv 0\;(\hbox{mod $n_{i_2}$})$. Hence, by \ref{eq:modmaster}, $\pi_0(a) + \pi_0(c)-2\pi_0(b)$ is a non-zero multiple of $n_{i_2}>2$.
\\
\\
{\sc Case 2:} $(a_i, \, b_i, \, c_i)$ is non-trivial for every $i$.
\\
\\
For any $t \in \{0, \pm 1, \pm 2\}$ we have that, by choice of $\pi_1, \dots \pi_6$, there exists an $i$ such that $\pi_i$ destroys the pattern $s\mapsto t$. Hence, we have $\pi_0(a)+\pi_0(c)-2\pi_0(b) \not \equiv t\;(\hbox{mod $\sqrt{n_0}$})$ as, by \ref{eq:modmaster}, they are not even congruent modulo $n_i$.
\\
\\
This completes the proof that $\pi_0$ destroys (1, 2)-almost APs.
\\
\par To prove Theorem \ref{thm:main}(i), it thus remains to show how to use the master permutation $\pi_0$ to construct an AP-destroying permutation of $\mathbb{Z}_n$ for every $n \geq n_0$. To begin with, let $m, \, n$ be any positive integers and write $n = k \cdot m + l$, where $0 \leq l < m$. Place the numbers $0,\,1,\,\dots,\,n-1$ clockwise around a circle, and divide them up into consecutive blocks $B_0,\,\dots,\,B_{m - 1}$, each of which has size $k$ or $k+1$. Thus there will be exactly $l$ blocks of size $k+1$. Let $\beta(x)$ denote the number of the block containing $x$, i.e.: $x \in B_{\beta(x)}$. We make two claims:
\\
\\
{\sc Claim 1:} If $k \geq m$ then no matter which blocks have size $k+1$, if $(a, \, b, \, c)$ is an AP modulo $n$, then $\beta(a) + \beta(c) - 2\beta(b) \in \{0, \, \pm 1, \, \pm 2\} \; ({\hbox{mod $m$}})$. 
\\
\\
To see this, consider a ``worst case'' where $a = 0$ and $b$ is the furthest clockwise (last) element of block $B_i$, for some $0 \leq i < m/2$ and such that $2b < n$. Then $k(i+1)-1 \leq b \leq (k+1)(i+1) - 1$ and so $2k(i+1) - 2 \leq 2b \leq 2(k+1)(i+1) - 2$. For the claim to hold, we need $2b$ to lie in one of the blocks $B_{2i-2}, \, \dots, \, B_{2i+2}$. The last element of $B_{2i-3}$ is at most $(k+1)(2i-2)$, while the first element of $B_{2i+3}$ is at least $k(2i+3)$. Hence the claim holds provided
\begin{equation}\label{eq:ineqs}
(k+1)(2i-2) < 2k(i+1) - 2 \;\;\;\; {\hbox{and}} \;\;\;\;
2(k+1)(i+1)-2 < k(2i+3).
\end{equation} 
Both inequalities are easily checked to hold provided $k \geq m$. The symmetric ``worst case'' where $a$ is the last and $b$ the first element in their respective blocks is handled similarly.
\\
\\
{\sc Claim 2:} For any $n$, if numbers are placed in blocks according to $\beta(x) := \lfloor mx/n \rfloor$, then every block has size $k$ or $k+1$ and, for any $(a, \, b, \, c)$ an AP modulo $n$, one has the stronger conclusion that 
$\beta(a) + \beta(c) - 2\beta(b) \in \{0, \, \pm 1\} \; ({\hbox{mod $m$}})$. 
\\
\\
This claim is easily verified by plugging in the formula for $\beta(x)$. 
\\
\par Now suppose $n \geq n_0$. Write $n = k \cdot \sqrt{n_0} + l$, where $0 \leq l < \sqrt{n_0}$. Imagine the numbers $0,\,1,\,\dots,\,n-1$ placed clockwise around a circle. We shall describe a rearrangement of this circular string of $n$ numbers such that, if $\pi(x)$ denotes the location of the number $x$ after the rearrangement, then $\pi$ will be an AP-destroying permutation of 
$\mathbb{Z}_n$. 
\par Firstly, divide the $n$ numbers into consecutive clockwise blocks $B_0,\,\dots,\,B_{\sqrt{n_0} - 1}$, such that $\beta(x) = \lfloor \frac{x \sqrt{n_0}}{n} \rfloor$. For each $i = 0,\,1,\,\dots,\,\sqrt{n_0} - 1$, let $\tau_i$ be a permutation of the elements of block $B_i$ which destroys APs, considering the elements of the block as lying in the group $\mathbb{Z}$ of ordinary integers. It follows from Proposition 2.3(ii) of \cite{H} that such permutations exist. Note that the $\tau_i$ will automatically also destroy APs modulo $n$. Our AP-destroying permutation $\pi$ of $\mathbb{Z}_n$ is gotten by first rearranging the blocks according to the master permutation $\pi_0$, and then applying $\tau_i$ within each block (or vice versa, the two operations commute). In other words, after applying $\pi$ to the circular arrangement of numbers, the blocks $B_{\pi_{0}^{-1}(0)}, \, B_{\pi_{0}^{-1}(1)}, \, \dots, \, B_{\pi_{0}^{-1}(\sqrt{n_0}-1)}$ appear in clockwise order and, within block $B_i$, its integer elements have been permuted according to $\tau_i$. Since $\pi_0$ is $(1,\,2)$-almost AP-destroying, it is easily deduced from Claims 1 and 2 that $\pi$ destroys APs modulo $n$. This completes the proof of Theorem 
\ref{thm:main}(i). 
\\
\par We now turn to part (ii) of the theorem and divide the proof into three steps. 
\\
\\
{\sc Step 1:} There exists a $(2, \, 2)$-almost AP-destroying permutation of $\mathbb{Z}_r$ for some $r$. 
\\
\\
\emph{Proof.} Let $\pi_1,\,\dots,\,\pi_6$ be as above. Using Table \ref{tab:data} and Lemma \ref{lem:normalise}, it is easy to check that we can also find permutations $\pi_7,\,\dots,\,\pi_{15}$ of $\mathbb{Z}_{n_{7}},\,\dots,\,\mathbb{Z}_{n_{15}}$ respectively which collectively destroy all of the patterns $s \mapsto t$, $s \in \{\pm 2\}$, $t \in \{0, \, \pm 1, \, \pm 2\}$. Hence, by a similar argument to above, there exists a $(2, \, 2)$-almost AP-destroying permutation $\chi_{r}$ of $\mathbb{Z}_r$, where $r = \prod_{i=1}^{15} n_i$.
\\
\\
{\sc Step 2:} For every $s, \, t \in \mathbb{N}_0$ there exists $r_{0}(s,\,t) \in \mathbb{N}$ and an $(s, \, t)$-almost AP-destroying permutation of $\mathbb{Z}_{r_{0}(s, \, t)}$. 
\\
\\ 
\emph{Proof.} Let $m$ be an integer such that there exists an AP-destroying permutation of $\mathbb{Z}_m$, and let $\chi_{m}$ be such a permutation. Identify $\mathbb{Z}_{N}$ with the set $\{0,\,1,\,\dots,\,N-1\}$ for any $N$. Let $r = r_{0}(2, \, 2)$ be as in Step 1. Then (see Lemma 3.5 of \cite{H}) the map $\chi_{rm}: \mathbb{Z}_{rm} \rightarrow \mathbb{Z}_{rm}$ given by 
\begin{equation}\label{eq:rm22}
\chi_{rm}(rx+y) = r \chi_{m}(x) + \chi_{r}(y), \;\; 0 \leq x < m, \;\; 0 \leq y < r,
\end{equation}
is easily seen to be $(2,\,2)$-almost AP-destroying, provided $m > 2$. Thus, by the already proven Theorem \ref{thm:main}(i), there exists a $(2, \, 2)$-almost AP-destroying permutation $\chi_{rm}$ of $\mathbb{Z}_{rm}$, for all sufficiently large $m$. Moreover, after a suitable translation, we can choose $\chi_{rm}$ so that $rm - 1$ is a fixed point. In this case, the restriction of $\chi_{rm}$ to $\{0,\,1,\,\dots,\,rm-2\}$ can be considered as a permutation of $\mathbb{Z}_{rm-1}$ and, since $\chi_{rm}$ was $(2, \, 2)$-almost AP-destroying, it is easily seen that this restriction is $(1, \, 1)$-almost AP-destroying. To summarise, we have shown that there is an infinite arithmetic progression $\mathcal{A}$, consisting of numbers congruent to $-1 \; ({\hbox{mod $r$}})$, such that there exists a $(1, \, 1)$-almost AP-destroying permutation of $\mathbb{Z}_n$ for all $n \in \mathcal{A}$. Furthermore, since the first term and common difference of $\mathcal{A}$ are relatively prime, there is an infinite subsequence $a_1,\,a_2,\dots$ of elements of $\mathcal{A}$ consisting of pairwise relatively prime numbers. This follows from Dirichlet's theorem, though it is actually trivial to prove, in a similar manner to Euclid's proof of the existence of infinitely many primes.
\par Now fix $s, \, t \in \mathbb{N}_0$. For a permutation of some $\mathbb{Z}_n$ to be $(s, \, t)$-almost AP-destroying, it just needs to destroy a finite number of patterns. Using Lemma \ref{lem:normalise} it follows that there exists an $i = i(s,\,t)$ and $(1,\,1)$-almost AP-destroying permutations $\chi_{j}$ of $\mathbb{Z}_{a_j}$, $j = 1,\,\dots,\,i$, which collectively destroy every pattern $s^{\prime} \mapsto t^{\prime}$, $|s^{\prime}| \leq s$, $|t^{\prime}| \leq t$. Then, by a construction similar to (\ref{eq:master}), we can construct an $(s,\,t)$-almost AP-destroying permutation of $\mathbb{Z}_{r_{0}(s,\,t)}$, where $r_{0}(s,\,t)= \prod_{j=1}^{i(s,\,t)} a_j$.
\\
\\
{\sc Step 3:} Theorem \ref{thm:main}(ii) holds.
\\
\\
\emph{Proof.} Clearly, it suffices to prove the theorem when $s=t$ and $t=0$ has already been dealt with. So fix $t > 0$. We claim the theorem holds with 
\begin{equation}\label{eq:nzerost}
n_{0}(s, \, t) = \left[ r_{0}(4t+7, \, 4t+7) \right]^2.
\end{equation}
To simplify notation, set $M := r_{0}(4t + 7, \, 4t+7)$ and fix $n \geq M^2$. Our task is to construct a $(t, \, t)$-almost AP-destroying permutation of $\mathbb{Z}_n$. Write $n = kM + l$, $0 \leq l < M$ and place the numbers $0,\,1,\,\dots,\,n-1$ clockwise around a circle. As before, we find it most convenient to describe our permutation in terms of a reearrangement of this circular string of $n$ numbers. The rearrangement will be broken down into 4 stages, of which stages 2 and 3 correspond to the procedure in the proof of part (i) of Theorem \ref{thm:main}, while stages 1 and 4 deal with the fact that $t > 0$. Stage 4 is essentially the ``reverse'' of Stage 1. 
\\
\\
\emph{Stage 1:} First divide the $n$ numbers into $M$ consecutive clockwise blocks $B_0,\,\dots,\,B_{M-1}$, each of size $\lfloor n/M \rfloor$ or $\lceil n/M \rceil$. Unlike in the proof of part (i), here it doesn't matter which blocks have which size, as we will only appeal in the end to Claim 1 from earlier. Let $\beta_{1}(x)$ denote the number of the block containing $x \in [0,\,n)$ at this point. 
\par Next, partition the blocks $B_i$ into $\lfloor \frac{M}{t+1} \rfloor$ ``superblocks'' $C_0,\,\dots,\,C_{\lfloor M/(t+1) \rfloor}$, each consisting of either $t+1$ or $t+2$ ordinary blocks, with the larger superblocks placed furthest clockwise from zero. Thus 
\begin{equation}\label{eq:superblocks}
C_0 = (B_0,\,\dots,\,B_{t}), \;\;\; C_{1} = (B_{t+1},\,\dots,\,B_{2t+1}), \;\;\; {\hbox{etc.}}
\end{equation}
Note that since $M$ is extremely large compared to $t$, there is no problem in making this subdivision of ordinary blocks.
Now rearrange the individual numbers in each superblock in such a way that, if the superblock contains $t+i$ ordinary blocks, $i \in \{1,\,2\}$, then after this rearrangement, the numbers inside any ordinary block will form an AP of common difference $t+i$. For example, consider the superblock $C_0$. There is a unique reordering $(i_0,\,i_1,\,\dots,\,i_t)$ of $(0,\,1,\,\dots,\,t)$ such that
$|B_{i_0}| \geq |B_{i_1}| \geq \dots \geq |B_{i_t}|$ and the indices are increasing as long as the block sizes are constant. After rearrangement, $B_{i_0}$ would contain $0,\,t+1,\,2(t+1),\dots$, $B_{i_1}$ would contain $1,\,t+2,\,2t+3,\dots$ and so on up to $B_{i_t}$ which would contain $t,\,2t+1,\dots$. 
\par Let $\beta_{2}(x)$ denote the (ordinary) block containing $x$ at this point and note that
\begin{equation}\label{eq:betadiff}
\left| \beta_{2}(x) - \beta_{1}(x) \right| \leq t+1,
\end{equation}
where plus one comes from the fact that some superblocks may contain $t+2$ ordinary blocks. 
\\
\\
\emph{Stage 2:} Choose a $(4t+7, \, 4t+7)$-almost AP-destroying permutation $\pi_0$ of $\mathbb{Z}_M$ and permute the ordinary blocks according to this - in other words, after applying $\pi_0$ the blocks $B_{\pi_{0}^{-1}(0)},\,\dots,\,B_{\pi_{0}^{-1}(M-1)}$ appear in clockwise order. Let $\mathcal{B}_i := B_{\pi_{0}^{-1}(i)}$ and let $\beta_{3} (x)$ denote the scripted block containing $x$ at this point. Thus 
\begin{equation}\label{eq:beta3}
\beta_{3} (x) = \pi_{0} (\beta_{2}(x)). 
\end{equation}
\emph{Stage 3:} Let $\tau^{0}, \, \tau^{1}$ be AP-destroying permutations
of $\left\{1,\,\dots,\, \lfloor \frac{n}{M} \rfloor \right\}$ and $\left\{1,\,\dots,\,\lceil \frac{n}{M} \rceil \right\}$ respectively, considered as subsets of $\mathbb{N}$. From Proposition 2.3(ii) of \cite{H} we know that such permutations exist. Given any set $S$ of integers which forms an AP of length $\lfloor n/M \rfloor$ (resp. $\lceil n/M \rceil$), it is obvious how to extract from $\tau^{0}$ (resp. $\tau^{1}$) an AP-destroying permutation of $S$. We perform such a permutation on each scripted block $\mathcal{B}_i$. 
\\
\\
\emph{Stage 4:} Divide the scripted blocks into superblocks in the same way as in (\ref{eq:superblocks}), thus
\begin{equation}
\mathcal{C}_0 = (\mathcal{B}_0,\,\dots,\,\mathcal{B}_{t}), \;\;\; \mathcal{C}_{1} = (\mathcal{B}_{t+1},\,\dots,\,\mathcal{B}_{2t+1}), \;\;\; {\hbox{etc.}}
\end{equation}
We then rearrange the numbers in each superblock in such a way that, for each block $\mathcal{B}_i$, the positions of its elements after rearrangement form an AP of common difference $|\mathcal{C}_0| \in \{t+1, \, t+2\}$. This is accomplished by reversing the procedure in Stage 1 - we hope it is clear what is meant by this and spare the reader further details. Let $\beta_{4}(x)$ be the number of the scripted block containing $x$ at this point and note that, analogous to (\ref{eq:betadiff}), one has
\begin{equation}\label{eq:scriptbetadiff}
\left| \beta_{4}(x) - \beta_{3}(x) \right| \leq t+1.
\end{equation}
\par Let $\pi: \mathbb{Z}_n \rightarrow \mathbb{Z}_n$ be the permutation defined by the rearrangement accomplished in Stages 1-4, that is, $\pi(x)$ denotes the location of the number $x$ in the string after the rearrangement. We claim that $\pi$ is $(t,\,t)$-almost AP-destroying. To see this, let $(a,\,b,\,c)$ be a triple of elements of $\mathbb{Z}_n$ satisfying
\par (I) $a, \, b, \, c$ not all equal,
\par (II) $a + c -2b \equiv \eta \; ({\hbox{mod $n$}})$ for some $\eta \in [-t, \, t]$
\\
and consider two cases:
\\
\\
{\sc Case 1:} $a, \, b, \, c$ all lie in the same ordinary block $B_i$ at the end of Stage 1. 
\\
\\
Since $a, \, b, \, c$ are not all equal and the numbers in $B_i$, as it stands after Stage 1, form an AP with common difference strictly greater than $t$, property (II) can only hold if $a, \, b, \, c$ form a non-trivial AP modulo $n$. But since the appropriate $\tau^{j}$ destroys APs of integers, and hence also APs modulo $n$ since $n$ is much larger than $M$, we see that the locations of $a, \, b, \, c$ will not form an AP modulo $n$ after Stage 3. But they will still lie in the same scripted block hence, after Stage 4, $\pi(a)+\pi(c)-2\pi(b)$ must be a non-zero multiple of $t+i$, $i \in \{1,\,2\}$. In particular, $\pi(a) + \pi(c) -2\pi(b) \; ({\hbox{mod $n$}}) \not\in [-t, \, t]$. 
\\
\\
{\sc Case 2:} $a, \, b, \, c$ are not all in the same ordinary block upon completion of Stage 1.
\\
\\
By definition, what we're assuming in this case is that $\beta_{2}(a), \, \beta_{2}(b), \, \beta_{2}(c)$ are not all equal. If $a, \, b, \, c$ formed an AP modulo $n$ then it would follow from Claim 1 in the proof of part (i) of Theorem \ref{thm:main} that 
\begin{equation}\label{eq:firstest}
\beta_{1}(a) + \beta_{1}(c) - 2\beta_{1}(b) \; ({\hbox{mod $M$}}) \in [-2, \, 2].
\end{equation}
Given that (II) holds and that each ordinary block has size greater than $t$, we can at least be sure that
\begin{equation}\label{eq:secondest}
\beta_{1}(a) + \beta_{1}(c) - 2 \beta_{1}(b) \; ({\hbox{mod $M$}}) \in [-3, \, 3].
\end{equation}
Combined with (\ref{eq:betadiff}) it follows that
\begin{equation}\label{eq:thirdest}
\beta_{2}(a) + \beta_{2}(c) - 2\beta_{2}(b) \; ({\hbox{mod $M$}}) \in [-(4t+7), \, 4t+7].
\end{equation}
But the permutation $\pi_0$ is $(4t+7, \, 4t+7)$-almost AP-destroying and thus, by (\ref{eq:beta3}), 
\begin{equation}\label{eq:afterest}
\beta_{3}(a) + \beta_{3}(c) - 2\beta_{3}(b) \; ({\hbox{mod $M$}}) \not\in [-(4t+7), \, 4t+7].
\end{equation}
By (\ref{eq:scriptbetadiff}), this implies in turn that
\begin{equation}\label{eq:diffover3}
\beta_{4}(a) + \beta_{4}(c) - 2 \beta_{4}(b) \; ({\hbox{mod $M$}}) \not\in [-3, \, 3].
\end{equation}
Since the scripted blocks still have size greater than $t$, it follows from Claim 1 on page 4 that $\pi(a) + \pi(c) - 2\pi(b) \; ({\hbox{mod $n$}}) \not\in [-t, \, t]$, as desired. 
\qed

\setcounter{equation}{0}

\section{Proof of Theorem \ref{thm:primes}}\label{sect:pfprimes}

Let $p$ be a prime. We denote by $\mathcal{R}_p$ (resp. $\mathcal{N}_p$)
the collection of quadratic residues (resp. non-residues) modulo $p$. We will
be slightly abusive in this context and use the same notations to denote 
subsets of $\mathbb{Z}_p$ and of $\mathbb{Z}$. Hence, as subsets of 
$\mathbb{Z}_p$ one has
\begin{equation}\label{eq:residues}
\mathcal{R}_p = \{x^2 : x \in \mathbb{Z}_p \}, \;\;\;\; \mathcal{N}_p = 
\mathbb{Z}_p \backslash \mathcal{R}_p,
\end{equation}
whereas, as subsets of $\mathbb{Z}$,
\begin{equation}\label{eq:intresidues}
\mathcal{R}_p = \left\{x \in \mathbb{Z} : \left( \frac{x}{p} \right) \in 
\{0,\,1\} \right\}, \;\;\;\; \mathcal{N}_p = 
\mathbb{Z} \backslash \mathcal{R}_p = \left\{x \in \mathbb{Z} : \left( 
\frac{x}{p} \right) = -1 \right\}.
\end{equation}

\begin{lemma}\label{lem:othree}
Let $p$ be a prime such that $p \equiv 3 \; ({\hbox{mod $8$}})$. Then 
both $-1$ and $2$ are in $\mathcal{N}_p$.
\end{lemma}

\begin{proof}
This is elementary number theory. That $-1$ is not a square mod $p$ follows
from Lagrange's theorem for groups. That $2$ is not a square follows from
Gauss' Lemma.
\end{proof}

\begin{lemma}\label{lem:ofour}
Let $p > 3$ be a prime. Then there exists an integer $\xi$ such that both
$\xi$ and $\xi - 1$ lie in $\mathcal{N}_p$.
\end{lemma}

\begin{proof}
This follows immediately from the fact that $|\mathcal{N}_p| = |\mathcal{R}_p| - 1$ and $\{0,\,1,\,4\} \subseteq \mathcal{R}_p$.
\end{proof}

\begin{lemma}\label{lem:ofive}
Let $p > 2$ be a prime, and let $a,\,b,\,c$ be integers not divisible by $p$. 
Then $(0,\,0)$ is the only solution in $\mathbb{Z}_p$ to the 
congruence $ax^2 + bxy + cy^2 \equiv 0 \; ({\hbox{mod $p$}})$ if and only
if $b^2 - 4ac \in \mathcal{N}_p$.
\end{lemma}

\begin{proof}
Elementary. Suppose we have a solution with $y \not\equiv 0$. Then 
completion of squares gives 
\begin{equation}\label{eq:disc}
xy^{-1} \equiv (2a)^{-1} \left( -b \pm \sqrt{b^2 - 4ac} \right),
\end{equation}
which is meaningful if and only if $b^2 - 4ac \in \mathcal{R}_p$. 
\end{proof}

We are now ready to prove Theorem \ref{thm:primes}. Let $p > 3$ be a prime congruent to 
$3 \; ({\hbox{mod $8$}})$ and let $\xi$ be any integer satisfying the
conditions of Lemma \ref{lem:ofour}.
Define the map $f: \mathbb{Z}_p \rightarrow \mathbb{Z}_p$ as follows:
\begin{equation}\label{eq:f}
f(x) = \left\{ \begin{array}{lr} x^2, & {\hbox{if $x \in \{0,\,2,\, 4,\,\dots,\,p-1\}$}}, 
\\ \xi x^2, & {\hbox{if $x \in \{1,\,3,\,5,\, \dots, \,p-2\}$}}. \end{array} \right.
\end{equation}
Note that $f$ is one-to-one. Let $(a,\, b,\, c)$ be an arithmetic 
progression modulo $p$. Denote $a := x$, $b := x+y$, $c := x+2y$.
We need to show that
\begin{equation}\label{eq:apf}
f(x) + f(x+2y)\equiv 2f(x+y) \Rightarrow y\equiv 0.
\end{equation}
Denote also $\mathbb{E} := \{0,\, 2,\, 4,\, \dots, \, p-1\}$ and 
$\mathbb{O} := \{1,\, 3,\, 5,\, \dots, \, p-2\}$. We 
do a case-by-case analysis. 
\\
\\
{\sc Case 1:} $a,\, b,\, c \in \mathbb{E}$. 
\\
\\
The congruence in (\ref{eq:apf}) becomes 
\begin{equation}\label{eq:cong}
x^2 + (x+2y)^2 \equiv 2 (x+y)^2,
\end{equation}
which reduces to $2y^2 \equiv 0$, hence $y \equiv 0$ since $p > 2$. Thus
(\ref{eq:apf}) holds in this case.
\\
\\
The case when $a,\, b,\, c \in \mathbb{O}$ is completely analogous to Case 1.
\\
\\
{\sc Case 2:} $a,\, b \in \mathbb{E}$, $c \in \mathbb{O}$.
\\
\\
The congruence in (\ref{eq:apf}) becomes
\begin{equation}\label{eq:cong2}
x^2 + \xi (x+2y)^2 \equiv 2(x+y)^2,
\end{equation}
which can be expanded as
\begin{equation}\label{eq:exp2}
(\xi -1) x^2 + 4(\xi - 1)xy + 2(2\xi -1) y^2 \equiv 0.
\end{equation}
By the choice of $\xi$ we know that $\xi-1 \in \mathcal{N}_p$, so
in particular $\xi -1 \not\equiv 0$. Also, $2\xi -1 \not\equiv 0$, 
for otherwise we would have $\xi - 1 \equiv -1/2$, contradicting 
Lemma \ref{lem:othree} and the fact that $\xi - 1 \in \mathcal{N}_p$. 
Hence all the 
coefficients in the binary quadratic form in (\ref{eq:exp2}) 
are non-zero modulo $p$,
so we can apply Lemma \ref{lem:ofive} and deduce that there is no solution with
$y \not\equiv 0$ if and only if 
$[4(\xi-1)]^2 - 8(\xi-1)(2\xi-1) \in \mathcal{R}_p$. But
\begin{equation}\label{eq:shit2}
[4(\xi - 1)]^2 - 8(\xi-1)(2\xi -1) = [4(\xi-1)]^2 
\cdot \left( \frac{-1}{2(\xi -1)} \right),
\end{equation}
hence this lies in $\mathcal{R}_p$ if and only if $\frac{-1}{2(\xi-1)}$
does so. But the latter contradicts Lemma \ref{lem:othree} 
and the choice of $\xi$. 
\\
\\
The case when $a \in \mathbb{O}$, $b,\, c \in \mathbb{E}$ is
completely analogous to Case 2. 
\\
\\
{\sc Case 3:} $a,\, c \in \mathbb{E}$, $b \in \mathbb{O}$.
\\
\\
The congruence in (\ref{eq:apf}) becomes
\begin{equation}\label{eq:cong3}
x^2 + (x+2y)^2 \equiv 2\xi (x+y)^2,
\end{equation}
which can be expanded as
\begin{equation}\label{eq:exp3}
(\xi -1) x^2 + 2(\xi - 1)xy + (\xi -2) y^2 \equiv 0.
\end{equation}
Once again, all the coefficients are non-zero modulo $p$,
so we can apply Lemma \ref{lem:ofive} and deduce that there is no solution with
$y \not\equiv 0$ if and only if $[2(\xi-1)]^2 - 4(\xi-1)(\xi-2) \in
\mathcal{R}_p$. But
\begin{equation}\label{eq:shit3}
[2(\xi - 1)]^2 - 4(\xi-1)(\xi -2) = [2(\xi-1)]^2 
\cdot \left( \frac{1}{\xi -1} \right),
\end{equation}
hence this lies in $\mathcal{R}_p$ if and only if $\frac{1}{\xi-1}$
does so. But the latter contradicts the choice of $\xi$. 
\\
\\
{\sc Case 4:} $a,\, b \in \mathbb{O}$, $c \in \mathbb{E}$.
\\
\\
The congruence in (\ref{eq:apf}) becomes
\begin{equation}\label{eq:cong4}
\xi x^2 + (x+2y)^2 \equiv 2\xi (x+y)^2,
\end{equation}
which can be expanded as
\begin{equation}\label{eq:exp4}
(\xi -1) x^2 + 4(\xi - 1)xy + 2(\xi -2) y^2 \equiv 0.
\end{equation}
The coefficients are all non-zero modulo $p$ so,
by Lemma \ref{lem:ofive}, there is no solution with
$y \not\equiv 0$ if and only if $[4(\xi-1)]^2 - 8(\xi-1)(\xi-2) \in
\mathcal{R}_p$. But
\begin{equation}\label{eq:shit4}
[4(\xi - 1)]^2 - 8(\xi-1)(\xi -2) = [4(\xi-1)]^2 
\cdot \left( \frac{\xi}{2(\xi -1)} \right),
\end{equation}
hence this lies in $\mathcal{R}_p$ if and only if $\frac{\xi}{2(\xi-1)}$
does so. Once again, this contradicts Lemma \ref{lem:othree} 
and the choice of $\xi$. 
\\
\\
The case when $a \in \mathbb{E}$, $b,\, c \in \mathbb{O}$ is
completely analogous to Case 4. 
\\
\\
{\sc Case 5:} $a,\, c \in \mathbb{O}$, $b \in \mathbb{E}$.
\\
\\
The congruence in (\ref{eq:apf}) becomes
\begin{equation}\label{eq:cong5}
\xi x^2 + \xi (x+2y)^2 \equiv 2(x+y)^2,
\end{equation}
which can be expanded as
\begin{equation}\label{eq:exp5}
(\xi -1) x^2 + 2(\xi - 1)xy + (2\xi -1) y^2 \equiv 0.
\end{equation}
The coefficients are still non-zero modulo $p$ so,
by Lemma \ref{lem:ofive}, there is no solution with
$y \not\equiv 0$ if and only if $[2(\xi-1)]^2 - 4(\xi-1)(2\xi-1) \in
\mathcal{R}_p$. But
\begin{equation}\label{eq:shit5}
[2(\xi - 1)]^2 - 4(\xi-1)(2\xi -1) = [2(\xi-1)]^2 
\cdot \left( \frac{-\xi}{\xi -1} \right),
\end{equation}
hence this lies in $\mathcal{R}_p$ if and only if $\frac{-\xi}{\xi-1}$
does so. Once again we have a contradiction to 
Lemma \ref{lem:othree} and the choice of $\xi$. This covers all possible cases and completes the proof of Theorem \ref{thm:primes}.
\qed

\setcounter{equation}{0}

\section{Final remarks}\label{sect:final}

We don't know of any reason to suspect that Conjecture C of \cite{H} may be false. However, a full proof of it remains a challenging problem, given that the value of $n_0$ in Theorem \ref{thm:main}(i) still leaves a brute-force computational attack way out of reach. A less ambitious goal could be to see how small $n_0$ can be made using the ideas of this paper.
\par An obvious approach to Conjecture C is to consider a uniformly random permutation $\pi$ of $\mathbb{Z}_n$ and let $X$ be the number of non-trivial APs not destroyed by $\pi$. It is easy to check that $\mathbb{E}[X] =  \Theta (n)$. As a first hypothesis, one might guess that the events that individual APs are not destroyed by $\pi$ are almost independent and hence that $X$ is approximately Poisson distributed. In that case, the proportion of AP-destroying permutations of $\mathbb{Z}_n$ would decrease exponentially with $n$. A priori, this approach could still only yield Conjecture C for $n$ sufficiently large, but probably with a value of $n_0$ which is sufficiently small to complete the proof by direct computation. However, it remains an open question whether this intuition can be made rigorous. A successful moment analysis \emph{was} carried out in \cite{JS} for permutations destroying APs of length $4$, with the important point being that the expected number of non-trivial $4$-term APs not destroyed by a random permutation of $\mathbb{Z}_n$ is $O(1)$. 
This provides some additional circumstantial evidence for Conjecture C. 
\par An AP of length $4$ is a common solution to the pair of linear equations $x_1 - 2x_2 + x_3 = 0$, $x_2 - 2x_3 + x_4 = 0$. It is natural to ask the following more general question:

\begin{question}\label{quest:genequation}
Let $k,\,m \in \mathbb{N}$ and let $\mathcal{L}_{i}(x_1,\,\dots,\,x_k) = 0$, $i = 1,\,\dots,\,m$, be linear equations with integer coefficients. When is it the case that there is an $n_0 = n_0 (\mathcal{L}_1,\,\dots,\,\mathcal{L}_m)$ such that there exists a permutation of $\mathbb{Z}_n$ destroying all non-trivial solutions (as defined in \cite{R}) to $\mathcal{L}_1 = \dots = \mathcal{L}_m = 0$ for all $n \geq n_0$ ?
\end{question}

Indeed, it seems reasonable even to ask this kind of question for general polynomial equations, not just linear ones. Here we confine further speculation to the case of a single linear equation $a_0 + \sum_{i=1}^{k} a_i x_i = 0$, $a_i \neq 0 \; \forall \, i > 0$. Recall that the equation is said to be \emph{(translation) invariant} if $a_0 = \sum_{i=1}^{k} a_i = 0$. One can make the following straightforward observations:
\\
\par (i) Suppose $a_0 \neq 0$. If $n > 2|a_0|$, then there exists a unit $u \in \mathbb{Z}_{n}^{\times}$ such that $ua_0 \not\equiv a_0 \; ({\hbox{mod $n$}})$. Then the permutation $x \mapsto ux \; ({\hbox{mod $n$}})$ will destroy all solutions to the equation. 
\par (ii) Suppose $\sum_{i=1}^{k} a_i \neq 0$. If $n > \left| \sum_{i=1}^{k} a_i \right|$ then the translation $x \mapsto x+1 \; ({\hbox{mod $n$}})$ will destroy all solutions to the equation.
\par (iii) Suppose the equation is invariant and $k = 2$, so the equation reads
$a_1 (x_1 - x_2) = 0$, for some $a_1 \neq 0$. A permutation $\pi$ of $\mathbb{Z}_n$ will destroy all non-trivial solutions if and only if, whenever $x_1$ and $x_2$ are distinct but lie in the same congruence class modulo $\frac{n}{{\hbox{GCD}}(a_1,\,n)}$, then $\pi(x_1)$ and $\pi(x_2)$ lie in different congruence classes. Clearly, such a permutation exists for all $n \geq \left[{\hbox{GCD}}(a_1, \, n)\right]^2$. 
\\
\par
Hence, for a single linear equation, Question \ref{quest:genequation} is only interesting if the equation is invariant and $k \geq 3$. If we then consider a uniformly random permutation of $\mathbb{Z}_n$, it is easy to see that the expected number of non-trivial solutions not destroyed is $\Theta (n^{k-2})$. This suggests that permutations destroying all non-trivial solutions should exist when $k = 3$, but perhaps do not do so at all when $k > 3$. We therefore ask the following:

\begin{question}\label{quest:singleeq}
Let $\mathcal{L}(x_1,\,\dots,\,x_k) = a_0 + \sum_{i=1}^{k} a_i x_i = 0$, $a_i, \in \mathbb{Z}$, $a_i \neq 0 \; \forall \, i > 0$, be a linear equation. Is it true that the following statements are equivalent: 
\par (i) There is an $n_0 = n_{0}(\mathcal{L})$ such that, for every $n \geq n_0$, there exists a permutation $\pi$ of $\mathbb{Z}_n$ destroying all non-trivial solutions of $\mathcal{L} = 0$. 
\par (ii) Either the equation $\mathcal{L} = 0$ is variant, or it is invariant 
and $k \in \{2, \, 3\}$ ?
\end{question} 

As a final remark, note that in Propostion 2.3(i) of \cite{H} we proved that no permutation of any finite abelian group can destroy all non-trivial solutions to the Sidon equation $x_1 + x_2 - x_3 - x_4 = 0$. However, we do not see at this point how to modify that argument for equations in four or more variables in general. 



\end{document}